\documentclass[11pt,reqno]{amsart}
\usepackage{color}

\usepackage{amsfonts,amscd}
\usepackage{amssymb}
\usepackage{url}
\usepackage[english]{babel}
\usepackage{hyperref}

\theoremstyle{plain}
\newtheorem{theorem}                 {Theorem}      [section]

\newtheorem{corollary}    [theorem]  {Corollary}
\newtheorem{lemma}        [theorem]  {Lemma}
\newtheorem{proposition}  [theorem]  {Proposition}

\theoremstyle{definition}

\newtheorem{example}      [theorem]  {Example}

\newtheorem{remark}       [theorem]  {Remark}

\numberwithin{equation}{section}
\allowdisplaybreaks

\def \rn{{\mathbb R}}

\def \A{\mathcal A}
\def \B{\mathcal B}

\def \F{\mathcal F}

\def \H{\mathcal H}
\def \I{\mathcal I}
\def \K{\mathcal K}

\def \V{\mathcal V}

\def\nab#1#2{\hbox{$\nabla$\kern -.3em\lower 1.0 ex
		\hbox{$#1$}\kern -.1 em {$#2$}}}

\def \lb#1#2{[#1,#2]}

\def \g{\mathfrak{g}}

\def \k{\mathfrak{k}}

\def \m{\mathfrak{m}}

\def \SL2{\widetilde{\text{\bf SL}}_{2}(\rn)}

\begin{document}

\title[Almost Hermitian Lie Groups]
{Natural almost Hermitian structures on\\ 
conformally foliated 4-dimensional\\ 
Lie groups with minimal leaves}

\author{Emma Andersdotter Svensson}
\address{Mathematics, Faculty of Science\\
University of Lund\\
Box 118, Lund 221 00\\
Sweden}
\email{emmaasv@hotmail.com}

\author{Sigmundur Gudmundsson}
\address{Mathematics, Faculty of Science\\
	University of Lund\\
	Box 118, Lund 221 00\\
	Sweden}
\email{Sigmundur.Gudmundsson@math.lu.se}

\begin{abstract}
Let $(G,g)$ be a 4-dimensional Riemannian Lie group with a  2-dimensional left-invariant, conformal foliation $\F$ with minimal leaves. Let $J$ be an almost Hermitian structure on $G$ adapted to the foliation $\F$.  We classify such structures $J$ which are almost Kähler $(\A\K)$, integrable $(\I)$ or Kähler $(\K)$.  Hereby we construct several new multi-dimensional examples in each class.
\end{abstract}




\maketitle

\section{Introduction}\label{section-introduction}

The theory of {\it harmonic morphisms} was initiated by Jacobi, in the  $3$-dimensional Euclidean geometry, with his famous paper \cite{Jac} from 1848. In the late 1970s this was then generalised to Riemannian geometry in \cite{Fug} and \cite{Ish} by Fuglede and Ishihara, independently.  This has lead to a vibrant development which can be traced back in the standard reference \cite{Bai-Woo-book}, by Baird and Wood, and the regularly updated online bibliography \cite{Gud-bib}, maintained by the second author.  The following result of Baird and Eells gives the theory of harmonic morphisms, with values in a surface, a strong geometric flavour.

\begin{theorem}\cite{Bai-Eel}\label{theo:B-E}
Let $\phi:(M^m,g)\to (N^2,h)$ be a horizontally conformal submersion from a Riemannian manifold to a surface. Then $\phi$ is harmonic if and only if $\phi$ has minimal fibres.
\end{theorem}

In their work \cite{Gud-Sve-6}, the authors classify the $4$-dimensional Riemannian Lie groups $(G,g)$ equipped with a $2$-dimensional, conformal and left-invariant foliation $\F$ with minimal leaves.  Such a foliation is locally given by a submersive harmonic morphism into a surface.  For such a Riemannian Lie group there exist  natural almost Hermitian structures $J$ adopted to the foliation structure. 
\smallskip  

The purpose of this work is to classify these structures which are {\it almost Kähler} $(\A\K)$, {\it integrable} $(\I)$ or {\it Kähler} $(\K)$, see \cite{Gra-Her}.  Our  classification result is the following.

\begin{theorem}
Let $(G,g,J)$ be a 4-dimensional almost Hermitian Lie group equipped with a $2$-dimensional, conformal and left-invariant foliation $\F$ with minimal leaves. If a  natural Hermitian structure $J$, adapted to the foliation $\F$, is almost Kähler $(\A\K)$, integrable $(\I)$ or Kähler $(\K)$, then the corresponding Lie algebra $\g_k^{\A\K}$, $\g_k^{\I}$ or $\g_k^{\K}$ of $G$ is one of those given below $(k=1,2\dots,20)$.
\end{theorem}



\section{Two Dimensional Conformal Foliations $\F$ on $(G^4,g)$}
\label{section-the-foliations}

Let $(M,g)$ be a Riemannian manifold, $\V$ be an integrable distribution on $M$ and denote by $\H$ its orthogonal complementary distribution. As customary, we shall also by $\V$ and $\H$ denote the orthogonal projections onto the corresponding subbundles of the tangent bundle $TM$ of $M$ and denote by $\F$ the foliation tangent to $\V$. Then the second fundamental form for $\V$ is given by
$$B^\V(Z,W)=\frac 12\,\H(\nabla_ZW+\nabla_WZ)\qquad(Z,W\in\V),$$
while the second fundamental form for $\H$ satisfies 
$$B^\H(X,Y)=\frac{1}{2}\,\V(\nabla_XY+\nabla_YX)\qquad(X,Y\in\H).$$
The foliation $\F$ tangent to $\V$ is said to be {\it conformal} if there exits a vector field $V\in \V$ such that the second fundamental form $B^\H$ satisfies $$B^\H=g\otimes V$$ and
$\F$ is said to be {\it Riemannian} if $V=0$.
Furthermore, $\F$ is said to be {\it minimal} if $\text{trace}\ B^\V=0$ and
{\it totally geodesic} if $B^\V=0$. This is equivalent to the
leaves of $\F$ being minimal or totally geodesic submanifolds
of $M$, respectively.

It is well-known that the fibres of a horizontally conformal
map (resp.\ Riemannian submersion) give rise to a conformal foliation
(resp.\ Riemannian foliation). Conversely, the leaves of any
conformal foliation (resp.\ Riemannian foliation) are
locally the fibres of a horizontally conformal 
(resp.\ Riemannian) submersion, see \cite{Bai-Woo-book}.
\smallskip

\medskip

Let $(G,g)$ be a $4$-dimensional Lie group equipped with a left-invariant Riemannian metric $g$ and $K$ be a $2$-dimensional subgroup of $G$.  Let $\k$ and $\g$ be the Lie algebras of $K$ and $G$, respectively.  Let $\m$ be the two dimensional orthogonal complement of $\k$ in $\g$ with respect to the Riemannian metric $g$ on $G$.  By $\V$ we denote the integrable distribution generated by $\k$ and by $\H$ its  orthogonal distribution given by $\m$.  Further let $\F$ be the foliation of $G$ tangent to $\V$.  Let the set 
$$\B=\{X,Y,Z,W\}$$ 
be an orthonormal basis for the Lie algebra $\g$ of $G$, such that $Z,W$ generate the subalgebra $\k$ and $[W,Z]=\lambda\,W$ for some $\lambda\in\rn$. The elements $X,Y\in\m$ can clearly be chosen such that $\H [X,Y]=\rho\,X$ for some $\rho\in\rn$.  
\smallskip 

If the $2$-dimensional foliation $\F$ is conformal with minimal leaves then, following \cite{Gud-Sve-6}, the Lie bracket relations of $\g$ take the following form
\begin{eqnarray}\label{system}
	\lb WZ&=&\lambda W,\notag\\
	\lb ZX&=&\alpha X +\beta Y+z_1 Z+w_1 W\notag,\\
	\lb ZY&=&-\beta X+\alpha Y+z_2 Z+w_2 W,\\
	\lb WX&=&     a X     +b Y+z_3 Z-z_1W,\notag\\
	\lb WY&=&    -b X     +a Y+z_4 Z-z_2W,\notag\\
	\lb YX&=&     r X         +\theta_1 Z+\theta_2 W,\notag
\end{eqnarray}
with real structure coefficients. To ensure that these equations actually define a Lie algebra they must satisfy the Jacobi equations.  These are equivalent to a system of 14 second order homogeneous polynomial equations in 14 variables, see Appendix \ref{appendix}.  This has been solved in \cite{Gud-Sve-6} and the solutions give a complete classifications of such Lie algebras.  They form 20 multi-dimensional families which can be found therein.  

For the foliation $\F$ we state the following
easy result describing the geometry of the situation.

\begin{proposition}\label{prop-geometry}
Let $(G,g)$ be a 4-dimensional Lie group with Lie algebra $\g$ as above.  Then
\begin{enumerate}
\item[(i)] $\F$ is {\it totally geodesic} if and only if
$z_1=z_2=z_3+w_1=z_4+w_2=0$,
\item[(ii)] $\F$ is {\it Riemannian} if and only if $\alpha=a=0$, and
\item[(iii)] $\H$ is {\it integrable} if and only if $\theta_1=\theta_2=0$.
\end{enumerate}
\end{proposition}

\section{The Adapted Almost Hermitian Structures}

Let $(G,g)$ be a $4$-dimensional Riemannian Lie group as above. Further, let $J$ be a left-invariant almost Hermitian structure on the tangent bundle $TG$ adapted to the decomposition $TG=\V\oplus\H$ i.e. a bundle isomorphism $J:TG\to TG$ such that 
$$J^2=-\text{id},\ \ J(\V )=\V\ \ \text{and}\ \ J(\H)=\H.$$
The almost Hermitian structure $J$ is left-invariant and can therefore be fully described by its action on the orthonormal basis 
$$\B=\{X,Y,Z,Y\}$$ of the Lie algebra $\g$ of $G$. Here the vector fields $X,Y\in\H$ and $Z,W\in\V$ can be chosen such that 
\begin{equation}\label{equation-J}
J(X)=Y\ \ \text{and}\ \ J(Z)=W.
\end{equation}

\begin{lemma}\label{lemma-almost-Kaehler}
Let $(G,g,J)$ be a $4$-dimensional almost Hermitian Lie group as above.  Then the structure $J$ is {\it almost Kähler} $(\A\K)$ if and only if 
$$\theta_1=2\, a\ \ \text{and}\ \ \theta_2=-2\,\alpha.$$ 
\end{lemma}

\begin{proof} 
The Kähler form $\omega$ satisfies $\omega(E,F)=g(JE,F)$ for all $E,F\in\g$.  It is well known that $J$ is almost Kähler if and only if the corresponding Kähler form $\omega$ is closed i.e. $d\omega=0$.
\begin{eqnarray*}
d\omega(X,Y,Z)
&=&
-g(J\left[X,Y\right],Z)
-g(J\left[Y,Z\right],X)
-g(J\left[Z,X\right],Y)\\
&=&
-g(J(-rX-\theta_1 Z-\theta_2 W),Z)\\
& &-
g(J(\beta X-\alpha Y-z_2 Z-w_2 W),X)\\
& &-g(J(\alpha X +\beta Y+z_1 Z+w_1 W),Y)\\
&=&
-g(-rY-\theta_1 W+\theta_2 Z,Z)\\
& &
-g(\beta Y+\alpha X-z_2 W+w_2 Z,X)\\
& &
-g(\alpha Y -\beta X+z_1 W-w_1 Z,Y)\\
&=&
-\theta_2-2\,\alpha.
\end{eqnarray*}
The statement is then easily obtained by performing the corresponding computations in the other cases 
$d\omega(W,X,Y)$, $d\omega(Z,W,X)$ and $d\omega(Y,Z,W)$.
\end{proof}

\begin{lemma}\label{lemma-integrable}
Let $(G,g,J)$ be a $4$-dimensional almost Hermitian Lie group as above.  Then the almost complex structure $J$ is {\it integrable} $(\I)$ if and only if 
$$2z_1-z_4-w_2=0\ \ \text{and}\ \ 2z_2+z_3+w_1=0.$$ 
\end{lemma}
\begin{proof}
It is well known that $J$ is integrable if and only if the  corresponding skew-symmetric Nijenhuis tensor $N$ vanishes i.e. if 
\begin{equation*}
N(E,F) = [E,F] + J[JE,F] + J[E,JF] - [JE,JF]=0,
\end{equation*}	for all $E,F\in\g$.  It is obvious that $N(X,Y)=N(Z,W)=0$. Now 
\begin{eqnarray*}
N(X,Z)
&=&[X,Z] + J[JX,Z] + J[X,JZ] - [JX,JZ]\\
&=&[X,Z] + J[Y,Z] + J[X,W] - [Y,W]\\
&=&-\alpha X -\beta Y-z_1 Z-w_1 W+J(\beta X-\alpha Y-z_2 Z-w_2 W)\\
& &-J(    a X     +b Y+z_3 Z-z_1W)-b X+a Y+z_4 Z-z_2W\\
&=&-\alpha X -\beta Y-z_1 Z-w_1 W+\beta Y+\alpha X-z_2 W+w_2 Z\\
& &- a Y     +b X-z_3 W-z_1Z-b X+a Y+z_4 Z-z_2W\\
&=&-(2z_1-z_4-w_2)\,Z-(2z_2+z_3+w_1)\,W.
\end{eqnarray*}
	Similar calculations for $N(X,W)$, $N(Y,Z)$ and $N(Y,W)$ easily provide the result.
\end{proof}

\begin{corollary}\label{corollary-Kähler}
Let $(G,g,J)$ be a $4$-dimensional almost Hermitian Lie group as above.  Then the structure $J$ is Kähler $(\K)$ if and only if 
$$2z_1-z_4-w_2=0,\ \ 2z_2+z_3+w_1=0,\ \ \theta_1=2a\ \ \text{and}\ \ \theta_2=-2\,\alpha.$$ 
\end{corollary}
\begin{proof}
The almost Hermitian structure $J$ is Kähler if and only if it is almost Kähler and integrable. Hence the statement is an immediate consequence of Lemmas  \ref{lemma-almost-Kaehler} and \ref{lemma-integrable}.
\end{proof}

\begin{remark}\label{HyperbolicRmrk}
The $2$-dimensional Lie algebra $\k$ of the subgroup $K$ of $G$ has orthonormal basis $\{Z,W\}$, satisfying $[W,Z]=\lambda\,W$.  This means that its simply connected universal covering group $\widetilde K$ has constant Gaussian curvature $\kappa=-\lambda^2$.  Thus it is either a hyperbolic disk $H^2_\lambda$ or the flat Euclidean plane $\rn^2$.
\end{remark}

The $4$-dimensional Riemannian Lie groups $(G,g)$, with a left-invariant conformal foliation $\F$ with minimal leaves, were classified in \cite{Gud-Sve-6}.  The purpose of this work is to in investigate their adapted almost Hermitian structure $J$, given by (\ref{equation-J}), in each case.  Here we adopt the notation introduced in \cite{Gud-Sve-6}.


\section{Case (A) - ($\lambda\neq 0$ and $(\lambda-\alpha)^2+\beta^2\neq 0$)}

\begin{example}[$\g_{1}(\lambda,r,w_1,w_2)$]
\label{exam-A1}
This is a 4-dimensional family of {\it solvable} Lie algebras obtained by assuming that $r\neq0$, see \cite{Gud-Sve-6}. The Lie bracket relations are given by
\begin{eqnarray*}
\lb WZ&=&\lambda W,\\
\lb ZX&=&w_1 W,\\
\lb ZY&=&w_2 W,\\
\lb YX&=&r X+\frac{rw_1}{\lambda}W.
\end{eqnarray*}

\noindent
$(\mathcal{AK}):$
According to Lemma \ref{lemma-almost-Kaehler}, the adapted almost Hermitian structure $J$ is {\it almost K\"{a}hler} if and only if $\theta_1=2a$ and $\theta_2=-2\alpha$. We see from the bracket relations that $a=\theta_1=\alpha=0$, hence $\theta_2=0$. But $\lambda\theta_2=rw_1$ which implies that $rw_1=0$. Since $r\neq 0$ we conclude that the family is in the  almost K\"{a}hler class if and only if $w_1=0$. This provides the  3-dimensional family $\g_1^{\mathcal{AK}}(\lambda,r,w_2)$, given by
\begin{eqnarray*}
\lb WZ&=&\lambda W,\\
\lb ZY&=&w_2 W,\\
\lb YX&=& r X.
\end{eqnarray*}
Here the corresponding simply connected Lie groups are semidirect products $H^2_r\ltimes H^2_\lambda$ of hyperbolic disks, see Remark \ref{HyperbolicRmrk}.

\smallskip
\noindent		
$(\mathcal{I}):$
According to Lemma \ref{lemma-integrable} the structure $J$ is {\it integrable} if and only if $$2z_2+z_3+w_1=0\ \ \text{and}\ \ 2z_1-z_4-w_2=0.$$ Because $z_1=z_2=z_3=z_4=0$ we conclude that $w_1=w_2=0$. From this we yield the 2-dimensional family $\g_1^{\mathcal{I}}(\lambda,r)$ satisfying 
\begin{eqnarray*}
\lb WZ&=&\lambda W,\\
\lb YX&=& r X.
\end{eqnarray*}

\noindent		
$(\mathcal{K}):$
The family $\g_1^{\mathcal{K}}(\lambda,r)=\g_1^{\mathcal{I}}(\lambda,r)$ is contained in $\g_1^{\mathcal{AK}}(\lambda,r,w_2)$ and thus {\it Kähler}, according to Corollary \ref{corollary-Kähler}.  Here the corresponding simply connected Lie groups are direct products $H^2_r\ltimes H^2_\lambda$ of hyperbolic disks.
\end{example}


\begin{example}[$\g_{2}(\lambda,\alpha,\beta,w_1,w_2)$]
\label{exam-A2} 
Here we have the $5$-dimensional family of {\it solvable} Lie algebras with the bracket relations
\begin{eqnarray*}
\lb WZ&=&\lambda W,\\
\lb ZX&=&\alpha X+\beta Y+w_1 W,\\
\lb ZY&=&-\beta X+\alpha Y+w_2 W.
\end{eqnarray*}

\noindent
$(\mathcal{AK}):$
Since $\theta_1=\theta_2=a=0$, the family $\g_2$ is {\it almost Kähler} if and only if $\alpha=0.$ This yields a 4-dimensional family $\g_{2}^{\mathcal{AK}}(\lambda,\beta,w_1,w_2)$ with the Lie bracket relations
\begin{eqnarray*}
\lb WZ&=&\lambda W,\\
\lb ZX&=&\beta Y+w_1 W,\\
\lb ZY&=&-\beta X+w_2 W.
\end{eqnarray*}

\noindent
$(\mathcal{I}):$
Because $z_1=z_2=z_3=z_4=0$, the structure $J$ is {\it integrable}  if and only if $w_1=w_2=0$. This provides us with a 3-dimensional family
$\g_{2}^{\mathcal{I}}(\lambda,\alpha,\beta)$ given by
\begin{eqnarray*}
\lb WZ&=&\lambda W,\\
\lb ZX&=&\alpha X+\beta Y,\\
\lb ZY&=&-\beta X+\alpha Y.
\end{eqnarray*}
The corresponding simply connected Lie group is clearly  a semidirect product $H^2_\lambda\ltimes\rn^2$.

\smallskip	
\noindent	
$(\mathcal{K}):$
The family $\g_2$ is {\it Kähler} if and only if it is both almost K\"{a}hler and integrable i.e. if  $\alpha=w_1=w_2=0.$
Here we obtain the 2-dimensional family
$\g_{2}^{\mathcal{K}}(\lambda,\beta)$
given by
\begin{eqnarray*}
\lb WZ&=&\lambda W,\\
\lb ZX&=&\beta Y,\\
\lb ZY&=&-\beta X.
\end{eqnarray*}
Here the corresponding simply connected Lie group is a semidirect product $H^2_\lambda\ltimes\rn^2$.
\end{example}


\begin{example}[$\g_{3}(\alpha,\beta,w_1,w_2,\theta_2)$]
\label{exam-A3}
This is a $5$-dimensional family of {\it solvable} Lie algebras with $r=\theta_1=0,$ $\theta_2\neq 0$ and $\lambda=-2\alpha$. The bracket relations are given by 
\begin{eqnarray*}
\lb WZ&=&-2\alpha W,\\
\lb ZX&=&\alpha X+\beta  Y+w_1 W,\\
\lb ZY&=&-\beta  X+\alpha Y+w_2 W,\\
\lb YX&=&\theta_2W.
\end{eqnarray*}

\noindent
$(\mathcal{AK}):$
Here $a=\theta_1=0,$ $\lambda=-2\alpha\neq0$ and $\theta_2\neq0.$ Thus, the family $\g_3$ is {\it almost Kähler} if and only if $\theta_2=-2\alpha=\lambda\neq 0,$ which gives 
the $4$-dimensional family $\g_{3}^{\mathcal{AK}}(\alpha,\beta,w_1,w_2)$ with
\begin{eqnarray*}
\lb WZ&=&-2\alpha W,\\
\lb ZX&=&\alpha X+\beta  Y+w_1 W,\\
\lb ZY&=&-\beta  X+\alpha Y+w_2 W,\\
\lb YX&=&-2\alpha W.
\end{eqnarray*}

\noindent
$(\mathcal{I}):$ Because $z_1=z_2=z_3=z_4=0$, we see that the structure $J$ is {\it integrable} if and only if $w_1=w_2=0,$ and obtain the family $\g_{3}^{\mathcal{I}}(\alpha,\beta,\theta_2)$ with
\begin{eqnarray*}
\lb WZ&=&-2\alpha W,\\
\lb ZX&=&\alpha X+\beta  Y,\\
\lb ZY&=&-\beta  X+\alpha Y,\\
\lb YX&=&\theta_2W.
\end{eqnarray*}
	
\noindent	
($\mathcal{K}):$ 
The family $\g_3$ is {\it Kähler} if and only if $\theta_2=-2\alpha\neq 0$ and $ w_1=w_2=0.$ This provides $\g_{3}^\mathcal{K}(\alpha,\beta)$ with
\begin{eqnarray*}
\lb WZ&=&-2\alpha W,\\
\lb ZX&=&\alpha X+\beta  Y,\\
\lb ZY&=&-\beta  X+\alpha Y,\\
\lb YX&=&-2\alpha W.
\end{eqnarray*}
\end{example}


\section{Case (B) - ($\lambda\neq 0$ and $(\lambda-\alpha)^2+\beta^2=0$)}

\begin{example}[$\g_{4}(\lambda,z_2,w_1,w_2)$]\label{exam-B1}
Here we have the $4$-dimensional family of {\it solvable} Lie algebras with the bracket relations	
\begin{eqnarray*}
\lb WZ&=&\lambda W,\\
\lb ZX&=&\lambda X+w_1 W,\\
\lb ZY&=&\lambda Y+z_2Z+w_2 W,\\
\lb WY&=&-z_2W,\\
\lb YX&=&-z_2 X-\frac{z_2w_1}{\lambda}W.
\end{eqnarray*}

\noindent
$(\mathcal{AK}):$ From $\theta_1=2a=0$, $\alpha=\lambda\neq0$ and $\lambda\theta_2=-z_2 w_1$, we see that $\g_4$ is {\it almost Kähler} if and only if $\theta_2=-2\alpha=-2\lambda.$ This implies that $2\lambda^2=w_1z_2\neq 0$ and we yield the $3$-dimensional  $\g_{4}^{\mathcal{AK}}(\lambda,z_2,w_2)$ with the following Lie bracket relations
\begin{eqnarray*}
\lb WZ&=&\lambda W,\\
\lb ZX&=& \lambda X+\frac{2\lambda^2}{z_2} W,\\
\lb ZY&=&\lambda Y+z_2Z+w_2 W,\\
\lb WY&=&-z_2W,\\
\lb YX&=&-z_2 X-2\lambda W.
\end{eqnarray*}

\noindent
$(\mathcal{I}):$ 
Because $z_1=z_3=z_4=0$, the structure $J$ is {\it integrable} if and only if $2z_2+w_1=w_2=0$. This gives a 2-dimensional family $\g_{4}^{\mathcal{I}}(\lambda,z_2)$ with bracket relations
\begin{eqnarray*}
\lb WZ&=&\lambda W,\\
\lb ZX&=&\lambda X-2z_2 W,\\
\lb ZY&=&\lambda Y+z_2Z,\\
\lb WY&=&-z_2W,\\
\lb YX&=&-z_2 X+\frac{2z_2^2}{\lambda}W.
\end{eqnarray*}

\noindent
$(\mathcal{K}):$
Here the {\it Kähler} condition is never satisfied. This requires $\lambda^2+z_2^2=0$, which clearly has no real solutions for $\lambda\neq 0$. Hence $\g_{4}^{\mathcal{K}}=\{0\}$.
\end{example}


\section{Case (C) - ($\lambda=0$, $r\neq 0$ and $(a\beta-\alpha b)\neq 0$)}

\begin{example}[$\g_{5}(\alpha,a,\beta,b,r)$]\label{exam-C1}
This is the $5$-dimensional family of {\it solvable} Lie algebras satisfying the bracket relations
\begin{eqnarray*}
\lb ZX&=&\alpha X +\beta Y
-\frac{r(\beta b-\alpha a)}{2(\alpha b-a\beta)} Z
-\frac{r(\alpha^2-\beta^2)}{2(\alpha b-a\beta)} W,\\
\lb ZY&=&-\beta X+\alpha Y
-\frac{r(\alpha b+\beta a)}{2(\alpha b-a\beta)} Z
+\frac{r\alpha\beta}{(\alpha b-a\beta)} W,\\
\lb WX&=&     a X     +b Y
-\frac{r(b^2-a^2)}{2(\alpha b-a\beta)} Z
-\frac{r(\alpha a-\beta b)}{2(\alpha b-a\beta)}W,\\
\lb WY&=&    -b X     +a Y
-\frac{rab }{(\alpha b-a\beta)} Z
+\frac{r(\alpha b+\beta a)}{2(\alpha b-a\beta)}W,\\
\lb YX&=&     r X
+\frac{ar^2}{2(\alpha b-a\beta)} Z
-\frac{\alpha r^2}{2(\alpha b-a\beta)} W.
	\end{eqnarray*}	
	
\noindent
$(\mathcal{AK}):$
In this case we have the two relations
$$\theta_1=\frac{ar^2}{2(\alpha b-a\beta)} \ \ \text{and} \ \ \theta_2=-\frac{\alpha r^2}{2(\alpha b-a\beta)}.$$ 
They imply that the structure $J$ is {\it almost Kähler} if and only if
$$2a=\frac{ar^2}{2(\alpha b-a\beta)} \ \ \text{and} \ \ 2\alpha=\frac{\alpha r^2}{2(\alpha b-a\beta)}.$$
The condition $a\beta-\alpha b\neq0$ shows that not both $a$ and $\alpha$ can be zero, which implies that $r^2=4{(\alpha b-a\beta)}>0$.  If we substitute this into the above system, we then yield the following two solutions
\begin{eqnarray*}
\lb ZX&=&\alpha X +\beta Y
\mp\frac{(\beta b-\alpha a)}{\sqrt{\alpha b-a\beta}} Z
\mp\frac{(\alpha^2-\beta^2)}{\sqrt{\alpha b-a\beta}} W,\\
\lb ZY&=&-\beta X+\alpha Y
\mp\frac{(\alpha b+\beta a)}{\sqrt{\alpha b-a\beta}} Z
\pm\frac{2\alpha\beta}{\sqrt{\alpha b-a\beta}} W,\\
\lb WX&=&     a X     +b Y
\mp\frac{(b^2-a^2)}{\sqrt{\alpha b-a\beta}} Z
\mp\frac{(\alpha a-\beta b)}{\sqrt{\alpha b-a\beta}}W,\\
\lb WY&=&    -b X     +a Y
\mp\frac{2ab }{\sqrt{\alpha b-a\beta}} Z
\pm\frac{(\alpha b+\beta a)}{\sqrt{\alpha b-a\beta}}W,\\
\lb YX&=& \pm 2\sqrt{\alpha b-a\beta}\, X
+2a\, Z
-2\alpha\, W.
\end{eqnarray*}	
Here we have obtained two $4$-dimensional families $\g_{5}^{\A\K}(\alpha,a,\beta,b)^\pm$ satisfying the almost Kähler condition.

\smallskip 
\noindent
$(\mathcal{I}):$ 
Here we are assuming that $r\neq 0$, $a\beta-\alpha b\neq0$ and  have 
$$
z_1=-\frac{r(\beta b-\alpha a)}{2(\alpha b-a\beta)}, \ \
z_2=-\frac{r(\alpha b+\beta a)}{2(\alpha b-a\beta)}, \ \ z_3=-\frac{r(b^2-a^2)}{2(\alpha b-a\beta)},
$$ 
$$
z_4=-\frac{2rab }{2(\alpha b-a\beta)}, \ \ w_1=-\frac{r(\alpha^2-\beta^2)}{2(\alpha b-a\beta)},\ \ 
w_2=\frac{2r\alpha\beta}{2(\alpha b-a\beta)}.$$
From this we observe that $J$ is {\it integrable} if and only if 
$$
2(\alpha b+\beta a)
+b^2-a^2
+\alpha^2-\beta^2=0$$
and
$$\beta b-\alpha a
-ab
+\alpha\beta
=0.$$
A simple calculation shows that $J$ is integrable if and only if $a=\beta$ and $b=-\alpha$. This provides us with the $3$-dimensional family $\g_{5}^\mathcal{I}(\alpha,\beta,r)$ fulfilling the bracket relations 
\begin{eqnarray*}
	\lb ZX&=&\alpha X +\beta Y
	-\frac{r\alpha\beta}{(\beta^2+\alpha^2)} Z+\frac{r(\alpha^2-\beta^2)}{2(\beta^2+\alpha^2)} W,\\
	\lb ZY&=&-\beta X+\alpha Y
	-\frac{r(\alpha^2-\beta^2)}{2(\beta^2+\alpha^2)} Z-\frac{r\alpha\beta}{(\beta^2+\alpha^2)} W,\\
	\lb WX&=&     \beta X     -\alpha Y
	+\frac{r(\alpha^2-\beta^2)}{2(\beta^2+\alpha^2)} Z+\frac{r\alpha \beta}{(\beta^2+\alpha^2)}W,\\
	\lb WY&=&    \alpha X     +\beta Y
	-\frac{r\alpha\beta }{(\beta^2+\alpha^2)} Z+\frac{r(\alpha^2-\beta^2)}{2(\beta^2+\alpha^2)}W,\\
	\lb YX&=&     r X
	-\frac{\beta r^2}{2(\beta^2+\alpha^2)} Z+\frac{\alpha r^2}{2(\beta^2+\alpha^2)} W.
\end{eqnarray*}

\noindent		
$(\mathcal{K}):$ Combining the observations $r^2=4{(\alpha b-a\beta)}>0$ from $(\A\K)$ and $a=\beta$, $b=-\alpha$ from $(\I)$ we see that the {\it Kähler} condition is never satisfied for the Lie algebra $\g_5$, so $\g_{5}^{\mathcal{K}}=\{0\}$.
\end{example}


\section{Case (D) - ($\lambda=0$, $r\neq 0$ and $(a\beta-\alpha b)=0$)}

\begin{example}[$\g_{6}(z_1,z_2,z_3,r,\theta_1,\theta_2)$]\label{exam-D1}
Here we have a $6$-dimensional family of {\it solvable} Lie algebras with $z_1^2=-w_1z_3\neq 0$,
$$z_4 =\frac{z_3(r+2z_2)}{2z_1},\ \ w_1 = -\frac{z_1^2}{z_3}	\ \ \text{and}\ \ w_2 = \frac{z_1(r-2z_2)}{2z_3}.$$ 
In this case the Lie bracket relations are given by
\begin{eqnarray*}
\lb ZX&=&z_1 Z-\frac{z_1^2}{z_3} W,\\
\lb ZY&=&z_2 Z+\frac{z_1(r-2z_2)}{2z_3} W,\\
\lb WX&=&z_3 Z-z_1W,\\
\lb WY&=&\frac{z_3(r+2z_2)}{2z_1} Z-z_2W,\\
\lb YX&=&r X+\theta_1 Z+\theta_2 W.
\end{eqnarray*}

\noindent	
$(\mathcal{AK}):$ 
Because $a=\alpha=0$, the structure $J$ is {\it almost Kähler} if and only if $\theta_1=\theta_2=0$.  This provides the $4$-dimensional family $\g_{6}^{\mathcal{AK}}(z_1,z_2,z_3,r)$ with the bracket relations
\begin{eqnarray*}
\lb ZX&=&z_1 Z-\frac{z_1^2}{z_3} W,\\
\lb ZY&=&z_2 Z+\frac{z_1(r-2z_2)}{2z_3} W,\\
\lb WX&=&z_3 Z-z_1W,\\
\lb WY&=&\frac{z_3(r+2z_2)}{2z_1} Z-z_2W,\\
\lb YX&=&r X.
\end{eqnarray*}
Here the corresponding simply connected Lie group is a semidirect product $H^2_r\ltimes\rn^2$.

\smallskip
\noindent		
$(\mathcal{I}):$
We note that 
$$z_1\neq0, \ \ z_4=\frac{z_3(r+2z_2)}{2z_1}, \ \ w_1=-\frac{z_1^2}{z_3} \ \ \text{and} \ \ w_2=\frac{z_1(r-2z_2)}{2z_3}.$$ 
Here the structure $J$ is {\it integrable} if and only if 
$$z_1^2=(2z_2+z_3)z_3 \ \ \text{and} \ \
z_1^2(4z_3-r+2z_2)=z_3^2(r+2z_2).$$
By multiplying the first equation by $(4z_3-r+2z_2)$, we see that these two conditions imply 
\begin{equation*}
(2z_2+z_3)(4z_3-r+2z_2)=z_3(r+2z_2),
\end{equation*}
which can be rewritten as
\begin{equation*}
(z_2+z_3)(2(z_2+z_3)-r)=0.
\end{equation*}
If $z_2+z_3=0$, our earlier conditions give $z_1^2+z_3^2=0$, which is not possible for nonzero real constants. Thus we must have $z_2+z_3\neq0$. This leaves us with $r=2(z_2+z_3)$. We conclude that $J$ is {\it integrable} if and only if 
$$r=\frac{z_1^2+z_3^2}{z_3} \ \ \text{and} \ \ z_2=\frac{z_1^2-z_3^2}{2z_3}.$$   This leads to the following bracket relations for the 4-dimensional family
$\g_{6}^\mathcal{I}(z_1,z_3,\theta_1,\theta_2)$
		\begin{eqnarray*}
			\lb ZX&=&z_1 Z-\frac{z_1^2}{z_3} W,\\
			\lb ZY&=&\frac{z_1^2-z_3^2}{2z_3} Z+z_1 W,\\
			\lb WX&=&z_3 Z-z_1W,\\
			\lb WY&=&z_1 Z-\frac{z_1^2-z_3^2}{2z_3}W,\\
			\lb YX&=&\frac{z_1^2+z_3^2}{z_3} X+\theta_1 Z+\theta_2 W.
		\end{eqnarray*}
	
\noindent
$(\mathcal{K}):$
In this case the {\it K\"{a}hler} condition is fulfilled if and only if
$$\theta_1=\theta_2=0, \ \ r=\frac{z_1^2+z_3^2}{z_3} \ \ \text{and} \ \ z_2=\frac{z_1^2-z_3^2}{2z_3}.$$
Hence we have the 2-dimensional family $\g_{6}^\mathcal{K}(z_1,z_3)$ with bracket relations
		\begin{eqnarray*}
			\lb ZX&=&z_1 Z-\frac{z_1^2}{z_3} W,\\
			\lb ZY&=&\frac{z_1^2-z_3^2}{2z_3} Z+z_1 W,\\
			\lb WX&=&z_3 Z-z_1W,\\
			\lb WY&=&z_1 Z-\frac{z_1^2-z_3^2}{2z_3}W,\\
			\lb YX&=&\frac{z_1^2+z_3^2}{z_3} X.
		\end{eqnarray*}
Here the corresponding simply connected Lie group is a semidirect product $H^2_{r}\ltimes \rn^2$.
\end{example}


\begin{example}[$\g_{7}(z_2,w_1,w_2,\theta_1,\theta_2)$]
\label{exam-D2}
This is a $5$-dimensional family of {\it solvable} Lie algebras with $z_1=z_3=z_4=0$, $r=2z_2$ and $w_1\neq 0$. 
The bracket relations are
\begin{eqnarray*}
\lb ZX&=&w_1 W,\\
\lb ZY&=&z_2 Z+w_2 W,\\
\lb WY&=&-z_2 W,\\
\lb YX&=&2z_2 X+\theta_1 Z+\theta_2 W.
\end{eqnarray*}
	
\noindent
$(\mathcal{AK}):$ 
Because $\alpha=a=0$, the family belongs to the {\it almost Kähler} class if and only if $\theta_1=\theta_2=0$. This gives the  $3$-dimensional family $\g_{7}^{\mathcal{AK}}(z_2,w_1,w_2)$ with
\begin{eqnarray*}
\lb ZX&=&w_1 W,\\
\lb ZY&=&z_2 Z+w_2 W,\\
\lb WY&=&-z_2 W,\\
\lb YX&=&2z_2 X.
\end{eqnarray*}
Here the corresponding simply connected Lie group is a semidirect product $H^2_{2z_2}\ltimes\rn^2$.

\smallskip	
\noindent
$(\mathcal{I}):$
We have $z_1=z_3=z_4=0.$ Thus the structure $J$ is {\it integrable} if and only if $2z_2+w_1=w_2=0$ and we yield a 3-dimensional family $\g_{7}^{\mathcal{I}}(z_2,\theta_1,\theta_2)$
given by
		\begin{eqnarray*}
			\lb ZX&=&-2z_2 W,\\
			\lb ZY&=&z_2 Z,\\
			\lb WY&=&-z_2 W,\\
			\lb YX&=&2z_2 X+\theta_1 Z+\theta_2 W.
		\end{eqnarray*}
Note that $z_2\neq0$.

\smallskip
\noindent
$(\mathcal{K}):$
The family $\g_7$ is of the {\it Kähler} class if and only if $\theta_1=\theta_2=0$, $w_2=0$ and $w_1=-2z_2\neq0$ providing the  $1$-dimensional family
$\g_{7}^{\mathcal{K}}(z_2)$ with
		\begin{eqnarray*}
			\lb ZX&=&-2z_2 W,\\
			\lb ZY&=&z_2 Z,\\
			\lb WY&=&-z_2 W,\\
			\lb YX&=&2z_2 X.
		\end{eqnarray*}
The corresponding simply connected Lie group is clearly a semidirect product $H^2_{2z_2}\ltimes \rn^2$.
\end{example}


\begin{example}[$\g_{8}(z_2,z_4,w_2,r,\theta_1,\theta_2)$]
\label{exam-D3}
Here we have the $6$-dimensional family of {\it solvable} Lie algebras with bracket relations
\begin{eqnarray*}
\lb ZY&=&z_2 Z+w_2 W,\\
\lb WY&=&z_4 Z-z_2 W,\\
\lb YX&=&r X+\theta_1 Z+\theta_2 W.
\end{eqnarray*}

\noindent
$(\mathcal{AK}):$
Since $\alpha=a=0$, the family is almost K\"{a}hler if and only if $\theta_1=\theta_2=0,$ giving the $4$-dimensional family
$\g_{8}^{\mathcal{AK}}(z_2,z_4,w_2,r)$ with
		\begin{eqnarray*}
			\lb ZY&=&z_2 Z+w_2 W,\\
			\lb WY&=&z_4 Z-z_2 W,\\
			\lb YX&=&r X.
		\end{eqnarray*}
Here the corresponding simply connected Lie group is a semidirect product $H^2_r\ltimes\rn^2$.

\smallskip	
\noindent
$(\mathcal{I}):$
We have $z_1=z_3=w_1=0,$ so $J$ is integrable if and only if $z_2=z_4+w_2=0$, so we yield the family
$\g_{8}^{\mathcal{I}}(w_2,r,\theta_1,\theta_2)$ with
		\begin{eqnarray*}
			\lb ZY&=&w_2 W,\\
			\lb WY&=&-w_2 Z,\\
			\lb YX&=&r X+\theta_1 Z+\theta_2 W.
		\end{eqnarray*}

\noindent	
$(\mathcal{K}):$
The family is K\"{a}hler if and only if $\theta_1=\theta_2=z_2=z_4+w_2=0.$ 
We get the $2$-dimensional family $\g_{8}^{\mathcal{K}}(w_2,r)$ with
		\begin{eqnarray*}
			\lb ZY&=&w_2 W,\\
			\lb WY&=&-w_2 Z,\\
			\lb YX&=&r X.
		\end{eqnarray*}
The corresponding simply connected Lie group is clearly a semidirect product $H^2_{r}\ltimes \rn^2$
\end{example}


\begin{example}[$\g_{9}(z_2,z_3,z_4,\theta_1,\theta_2)$]
\label{exam-D4}
This is the $5$-dimensional family of {\it solvable} Lie algebras with $z_3\neq 0$ and bracket relations
	\begin{eqnarray*}
		\lb ZY&=&z_2 Z,\\
		\lb WX&=&z_3 Z,\\
		\lb WY&=&z_4 Z-z_2 W,\\
		\lb YX&=&-2z_2 X+\theta_1 Z+\theta_2 W.
	\end{eqnarray*}

\noindent	
$(\mathcal{AK}):$ 
Because $\alpha=a=0$, we see that the family is in the {\it almost K\"{a}hler} class if and only if $\theta_1=\theta_2=0,$ providing $\g_{9}^{\mathcal{AK}}(z_2,z_3,z_4)$ with
		\begin{eqnarray*}
			\lb ZY&=&z_2 Z,\\
			\lb WX&=&z_3 Z,\\
			\lb WY&=&z_4 Z-z_2 W,\\
			\lb YX&=&-2z_2 X.
		\end{eqnarray*}
		
\noindent	
$(\mathcal{I}):$ 
From $z_1=w_1=w_2=0,$ we find that $J$ is {\it integrable} if and only if $2z_2+z_3=z_4=0,$ from which we yield $\g_{9}^{\mathcal{I}}(z_2,\theta_1,\theta_2)$ with the relations
		\begin{eqnarray*}
			\lb ZY&=&z_2 Z,\\
			\lb WX&=&-2z_2 Z,\\
			\lb WY&=&-z_2 W,\\
			\lb YX&=&-2z_2 X+\theta_1 Z+\theta_2 W.
		\end{eqnarray*}

\noindent
$(\mathcal{K}):$
The family is {\it K\"{a}hler} if and only if
$\theta_1=\theta_2=0,$ $z_3=-2z_2$ and $z_4=0.$
Here we obtain the $1$-dimensional family $\g_{9}^{\mathcal{K}}(z_2)$ with
		\begin{eqnarray*}
			\lb ZY&=&z_2 Z,\\
			\lb WX&=&-2z_2 Z,\\
			\lb WY&=&-z_2 W,\\
			\lb YX&=&-2z_2 X.
		\end{eqnarray*}
The corresponding simply connected Lie group is clearly a semidirect product $H^2_{2z_2}\ltimes \rn^2$
\end{example}


\section{Case (E) - ($\lambda=0$, $r=0$ and $\alpha b-a\beta\neq 0$)}

\begin{example}[$\g_{10}(\alpha,a,\beta,b)$]\label{exam-E1}
Here we have the $4$-dimensional family of {\it solvable} Lie algebras with the bracket relations
	\begin{eqnarray*}
		\lb ZX&=&\alpha X+ \beta Y,\\
		\lb ZY&=&-\beta X+\alpha Y,\\
		\lb WX&=&     a X     +b Y,\\
		\lb WY&=&    -b X     +a Y.
	\end{eqnarray*}
Here the corresponding simply connected Lie group is a semidirect product $\rn^2\ltimes\rn^2$.

\smallskip		
\noindent	
$(\mathcal{AK}):$
Since $\theta_1=\theta_2=0$, we see that the family is {\it almost K\"{a}hler} if and only if $a=\alpha=0.$ But then the requirement $\alpha b-a\beta\neq0$ is not true, so this family can not be almost K\"{a}hler, so $\g_{10}^{\A\mathcal{K}}=\{0\}$.

\smallskip		
\noindent
$(\mathcal{I}):$ 
Because $z_1=z_2=z_3=z_4=w_1=w_2=0$, we can see that the structure $J$ is always {\it integrable} in this family, therefore  $\g_{10}^{\mathcal{I}}=\g_{10}(\alpha,a,\beta,b)$.

\smallskip		
\noindent
$(\mathcal{K}):$
The Lie algebra $\g_{10}$ is never in the {\it K\"{a}hler} class, hence $\g_{10}^{\mathcal{K}}=\{0\}$.
\end{example}


\section{Case (F) - ($\lambda=0$, $r=0$ and $\alpha b-a\beta= 0$)}
In this case our analysis divides into disjoint cases parametrized by
$\Lambda=(\alpha,a,\beta,b)$.  The variables are assumed to be zero if and only if they are marked by $0$. For example, if $\Lambda=(0,a,\beta,0)$ then the two variable $\alpha$
and $b$ are assumed to be zero and $a$ and $\beta$ to be non-zero.

\begin{example}[$\g_{11}(z_1,z_2,z_3,w_1,\theta_1,\theta_2)$]\label{exam-F1}
This is a $6$-dimensional family of {\it solvable} Lie algebras. Here $\Lambda=(0,0,0,0)$ and $z_1\neq 0$, giving the Lie bracket relations
	\begin{eqnarray*}
		\lb ZX&=&z_1 Z+w_1 W,\\
		\lb ZY&=&z_2 Z+\frac{z_2w_1}{z_1} W,\\
		\lb WX&=&z_3 Z-z_1W,\\
		\lb WY&=&\frac{z_2z_3}{z_1} Z-z_2 W,\\
		\lb YX&=&\theta_1 Z+\theta_2 W.
	\end{eqnarray*}

\noindent	
$(\mathcal{AK}):$
Because $\alpha=a=0$, we see that $\g_{11}$ is in the {\it almost K\"{a}hler} class if and only if $\theta_1=\theta_2=0.$ This gives a 4-dimensional family $\g_{11}^{\mathcal{AK}}(z_1,z_2,z_3,w_1)$
with
		\begin{eqnarray*}
			\lb ZX&=&z_1 Z+w_1 W,\\
			\lb ZY&=&z_2 Z+\frac{z_2w_1}{z_1} W,\\
			\lb WX&=&z_3 Z-z_1W,\\
			\lb WY&=&\frac{z_2z_3}{z_1} Z-z_2 W.
		\end{eqnarray*}
Here the corresponding simply connected Lie group is a semidirect product $\rn^2\ltimes\rn^2$.

\smallskip		
\noindent
$(\mathcal{I}):$
From the relations 
$$z_1\neq0, \ \ z_4=\frac{z_2 z_3}{z_1} \ \ \text{and} \ \ w_2=\frac{z_2 w_1}{z_1},$$ we see that the structure $J$ is {\it integrable} if and only if
$$z_3+w_1=-2z_2\ \ \text{and} \ \ 2z_1^2-z_2(z_3+w_1)=0.$$
This implies the impossible $z_1^2+z_2^2=0$, so the structure $J$ is never integrable, therefore $\g_{11}^{\mathcal{K}}=\{0\}$.

\smallskip		
\noindent
$(\mathcal{K}):$ 
This family $\g_{11}$ is never in the {\it K\"{a}hler} class, so $\g_{11}^{\mathcal{K}}=\{0\}$.
\end{example}


\begin{example}[$\g_{12}(z_3,w_1,w_2,\theta_1,\theta_2)$]
\label{exam-F2}
Here we have the $5$-dimensional family of {\it solvable} Lie algebras. Here $\Lambda=(0,0,0,0)$, $z_1=0$ and $w_1\neq 0$, which gives
$$z_2=0\ \ \text{and}\ \ z_4=\frac{z_3w_2}{w_1}$$
and the Lie bracket relations
	\begin{eqnarray*}
		\lb ZX&=&w_1 W,\\
		\lb ZY&=&w_2 W,\\
		\lb WX&=&z_3 Z,\\
		\lb WY&=&\frac{z_3w_2}{w_1} Z,\\
		\lb YX&=&\theta_1 Z+\theta_2 W.
	\end{eqnarray*}

\noindent	
$(\mathcal{AK}):$ 
Since $\alpha=a=0$, this family is {\it almost K\"{a}hler} if and only if $\theta_1=\theta_2=0.$ This yields the 3-dimensional family $\g^{\mathcal{AK}}_{12}(z_3,w_1,w_2)$ with the relations
		\begin{eqnarray*}
			\lb ZX&=&w_1 W,\\
			\lb ZY&=&w_2 W,\\
			\lb WX&=&z_3 Z,\\
			\lb WY&=&\frac{z_3w_2}{w_1} Z.
		\end{eqnarray*}
Here the corresponding simply connected Lie group is a semidirect product $\rn^2\ltimes\rn^2$.

\smallskip		
\noindent
$(\mathcal{I}):$ 
For $\g_{12}$ we have
$$z_1=z_2=0, \ \ z_4=\frac{z_3w_2}{w_1} \ \ \text{and} \ \ w_1\neq0.$$ We see that the structure $J$ is {\it integrable} if and only if $$z_3+w_1=\frac{z_3w_2}{w_1}+w_2=0,$$
or equivalently, $z_3=-w_1\neq 0$.  Here we yield the $4$-dimensional family $\g_{12}^{\mathcal{I}}(w_1,w_2,\theta_1,\theta_2)$
which has the Lie bracket relations
		\begin{eqnarray*}
			\lb ZX&=&w_1 W,\\
			\lb ZY&=&w_2 W,\\
			\lb WX&=&-w_1 Z,\\
			\lb WY&=&-w_2 Z,\\
			\lb YX&=&\theta_1 Z+\theta_2 W.
		\end{eqnarray*}

\noindent
$(\mathcal{K}):$
The family is {\it K\"{a}hler} if and only if $\theta_1=\theta_2=0$ and $z_3=-w_1\neq 0,$ which gives the  $2$-dimensional family
$\g_{12}^{\mathcal{K}}(w_1,w_2)$ with relations
		\begin{eqnarray*}
			\lb ZX&=&w_1 W,\\
			\lb ZY&=&w_2 W,\\
			\lb WX&=&-w_1 Z,\\
			\lb WY&=&-w_2 Z.
		\end{eqnarray*}
The corresponding simply connected Lie group is clearly a semidirect product $\rn^2\ltimes \rn^2$.
\end{example}


\begin{example}[$\g_{13}(z_3,z_4,\theta_1,\theta_2)$]
\label{exam-F3}
This is a $4$-dimensional family of {\it nilpotent} Lie algebras. Here $\Lambda=(0,0,0,0)$,  $z_1=w_1=0$ and $z_3\neq 0$, which gives $z_2=w_2=0$ and the solutions
	\begin{eqnarray*}
		\lb WX&=&z_3 Z,\\
		\lb WY&=&z_4 Z,\\
		\lb YX&=&\theta_1 Z+\theta_2 W.
	\end{eqnarray*}

\noindent	
$(\mathcal{AK}):$
Because $\alpha=a=0$, we see that $\g_{13}$ family is in the {\it almost K\"{a}hler} class if and only if $\theta_1=\theta_2=0$. We get the $2$-dimensional family $\g_{13}^{\mathcal{AK}}(z_3,z_4)$ with
		\begin{eqnarray*}
			\lb WX&=&z_3 Z,\\
			\lb WY&=&z_4 Z.
		\end{eqnarray*}
Here the corresponding simply connected Lie group is a semidirect product $\rn^2\ltimes\rn^2$.
	
\smallskip		
\noindent	
$(\mathcal{I}):$ 
The condition $z_1=z_2=w_1=w_2=0$ gives that $J$ is {\it integrable} if and only if $z_3=z_4=0$, contradicting $z_3\neq 0$.  Hence $\g_{13}^{\mathcal{I}}=\{0\}$
	
\smallskip		
\noindent	
$(\mathcal{K}):$ 
This family $\g_{13}$ is never in the {\it K\"{a}hler} class, so $\g_{13}^{\mathcal{K}}=\{0\}$.
\end{example}


\begin{example}[$\g_{14}(z_2,z_4,w_2,\theta_1,\theta_2)$]\label{exam-F4}
Here we have the $5$-dimensional family of {\it solvable} Lie algebras with $\Lambda=(0,0,0,0)$ and $z_1=z_3=w_1=0$ and the bracket relations
\begin{eqnarray*}
\lb ZY&=&z_2 Z+w_2 W,\\
\lb WY&=&z_4 Z-z_2 W,\\
\lb YX&=&\theta_1 Z+\theta_2 W.
\end{eqnarray*}

\noindent	
$(\mathcal{AK}):$
The condition $\alpha=a=0$ implies that this family is {\it almost K\"{a}hler} if and only if $\theta_1=\theta_2=0.$ We get the $3$-dimensional family $\g_{14}^{\mathcal{AK}}(z_2,z_4,w_2)$ given by
		\begin{eqnarray*}
			\lb ZY&=&z_2 Z+w_2 W,\\
			\lb WY&=&z_4 Z-z_2 W.
		\end{eqnarray*}
Here the corresponding simply connected Lie group is a semidirect product $\rn^2\ltimes\rn^2$.

\smallskip		
\noindent
$(\mathcal{I}):$
From $z_1=z_3=w_1=0$ we see that $J$ is {\it integrable} if and only if $z_2=0$ and $z_4+w_2=0.$ This gives a $3$-dimensional family
$\g_{14}^{\I}(w_2,\theta_1,\theta_2)$ with Lie bracket relations
		\begin{eqnarray*}
			\lb ZY&=&w_2 W,\\
			\lb WY&=&-w_2 Z,\\
			\lb YX&=&\theta_1 Z+\theta_2 W.
		\end{eqnarray*}

\noindent
$(\mathcal{K}):$
The family is {\it K\"{a}hler} if and only if $\theta_1=\theta_2=0,$ $z_2=0$ and $z_4+w_2=0.$ We get a $1$-dimensional family $\g_{14}^{\mathcal{K}}(w_2)$ with
		\begin{eqnarray*}
			\lb ZY&=&w_2 W,\\
			\lb WY&=&-w_2 Z.
		\end{eqnarray*}
The corresponding simply connected Lie group is clearly a semidirect product $\rn^2\ltimes \rn^2$.
\end{example}


\begin{example}[$\g_{15}(\alpha,w_1,w_2)$]
\label{exam-F5}
This is a $3$-dimensional family of {\it solvable} Lie algebras with $\Lambda=(\alpha,0,0,0)$, which gives
	\begin{equation*}
		z_1=z_2=z_3=z_4=\theta_1=\theta_2=0
	\end{equation*}
	and
	\begin{eqnarray*}
		\lb ZX&=&\alpha X+w_1 W,\\
		\lb ZY&=&\alpha Y+w_2 W.
	\end{eqnarray*}
		
\noindent
$(\mathcal{AK}):$
Here, $a=\theta_1=\theta_2=0$ and $\alpha\neq0,$ so we see that the {\it almost K\"{a}hler} is never fulfilled, hence $\g_{15}^{\mathcal{AK}}=\{0\}$.

\smallskip		
\noindent
$(\mathcal{I}):$ 
The conditions $z_1=z_2=z_3=z_4=0$ imply that $J$ is {\it integrable} if and only if $w_1=w_2=0.$ We yield a 1-dimensional family $\g_{15}^{\mathcal{I}}(\alpha)$ with
		\begin{eqnarray*}
			\lb ZX&=&\alpha X,\\
			\lb ZY&=&\alpha Y.
		\end{eqnarray*}
Here the corresponding simply connected Lie group is a semidirect product $\rn^2\ltimes\rn^2$.

\smallskip		
\noindent	
$(\mathcal{K}):$ 
Here the {\it K\"{a}hler} condition is never satisfied, so $\g_{15}^{\mathcal{K}}=\{0\}$.
\end{example}


\begin{example}[$\g_{16}(\beta,w_1,w_2,\theta_1,\theta_2)$]
\label{exam-F6}
Here we have a $5$-dimensional family of Lie algebras. They are {\it not solvable} in general.  Here $\Lambda=(0,0,\beta,0)$, which implies $z_1=z_2=z_3=z_4=0$. This gives
	\begin{eqnarray*}
		\lb ZX&=&\beta Y+w_1 W,\\
		\lb ZY&=&-\beta X+w_2 W,\\
		\lb YX&=&\theta_1 Z+\theta_2 W.
	\end{eqnarray*}

\noindent	
$(\mathcal{AK}):$
We have $\alpha=a=0,$ so this family is {\it almost K\"{a}hler} if and only if $\theta_1=\theta_2=0.$ We get the $3$-dimensional family
$\g_{16}^{\mathcal{AK}}(\beta,w_1,w_2)$ with
		\begin{eqnarray*}
			\lb ZX&=&\beta Y+w_1 W,\\
			\lb ZY&=&-\beta X+w_2 W.
		\end{eqnarray*}

\noindent
$(\mathcal{I}):$ 
We have $z_1=z_2=z_3=z_4=0,$ so $J$ is {\it integrable} if and only if $w_1=w_2=0,$ giving the $3$-dimensional family $\g_{16}^{\mathcal{I}}(\beta,\theta_1,\theta_2)$ with
		\begin{eqnarray*}
			\lb ZX&=&\beta Y,\\
			\lb ZY&=&-\beta X,\\
			\lb YX&=&\theta_1 Z+\theta_2 W.
		\end{eqnarray*}
			
\noindent
$(\mathcal{K}):$
The family $\g_{16}$ is {\it K\"{a}hler} if and only if $\theta_1=\theta_2=0$ and $w_1=w_2=0.$ We obtain the 1-dimensional family $\g_{16}^{\mathcal{K}}(\beta)$ with Lie bracket relations
		\begin{eqnarray*}
			\lb ZX&=&\beta Y,\\
			\lb ZY&=&-\beta X.
		\end{eqnarray*}
The corresponding simply connected Lie group is clearly a semidirect product $\rn^2\ltimes \rn^2$.
\end{example}


\begin{example}[$\g_{17}(\alpha,a,w_1,w_2)$]
\label{exam-F7}
This is a $4$-dimensional family of {\it solvable} Lie algebras  $\Lambda=(\alpha,a,0,0)$, which gives
$$
z_1=-\frac {aw_1}\alpha,\ \ 
z_2=-\frac {aw_2}\alpha,\ \ 
z_3=-\frac {a^2w_1}{\alpha^2},\ \ 
z_4=-\frac {a^2w_2}{\alpha^2},\ \
\theta_1=0,\ \ 
\theta_2=0.
$$
Thus the bracker relations are
	\begin{eqnarray*}
		\lb ZX&=&\alpha X-\frac {aw_1}\alpha Z+w_1 W,\\
		\lb ZY&=&\alpha Y-\frac {aw_2}\alpha Z+w_2 W,\\
		\lb WX&=&a X-\frac {a^2w_1}{\alpha^2} Z+\frac {aw_1}\alpha W,\\
		\lb WY&=&a Y-\frac {a^2w_2}{\alpha^2} Z+\frac {aw_2}\alpha W.
	\end{eqnarray*}

\noindent	
$(\mathcal{AK}):$
Because $\theta_1=\theta_2=0$, $\alpha\neq0$ and $a\neq0$, the {\it almost K\"{a}hler} condition is never satisfied, hence $\g_{17}^{\mathcal{AK}}=\{0\}$.

\smallskip		
\noindent
$(\mathcal{I}):$ 
From
$$z_1=-\frac {aw_1}\alpha, \ \ z_2=-\frac {aw_2}\alpha, \ \ z_3=-\frac {a^2w_1}{\alpha^2} \ \ \text{and} \ \ z_4=-\frac {a^2w_2}{\alpha^2}$$ 
we see that the structure $J$ is {\it integrable} if and only if
		\begin{equation}\label{eqg17}
			\left(\alpha^2- {a^2}\right)w_1=2a\alpha w_2 \ \ \text{and} \ \ \left(\alpha^2- {a^2}\right)w_2=-2a\alpha w_1.
		\end{equation}
Since $2a\alpha\neq0$, we can divide and  get
		\begin{equation*}
			w_2=\frac{\alpha^2-a^2}{2a\alpha}w_1
			=-\frac{(\alpha^2-a^2)^2}{4a^2\alpha^2}w_2.
		\end{equation*}
If $w_2\neq0$, then we yield the impossible $(\alpha^2+a^2)^2=0$.  Thus $J$ is integrable if and only if $w_1=w_2=0$. This gives the $2$-dimensional family $\g_{17}^{\mathcal{H}}(\alpha,a)$ with
		\begin{eqnarray*}
			\lb ZX&=&\alpha X,\\
			\lb ZY&=&\alpha Y,\\
			\lb WX&=&a X,\\
			\lb WY&=&a Y.
		\end{eqnarray*}
Here the corresponding simply connected Lie group is a semidirect product $\rn^2\ltimes\rn^2$.

\smallskip		
\noindent	
$(\mathcal{K}):$
The family $\g_{17}$ never satisfies the {\it K\"{a}hler} condition, so  $\g_{17}^{\mathcal{AK}}=\{0\}$.
\end{example}


\begin{example}[$\g_{18}(\beta,b,z_3,z_4,\theta_1,\theta_2)$]
\label{exam-F8}
Here we have a $6$-dimensional family of Lie algebras. They are {\it not solvable} in general. Here $\Lambda=(0,0,\beta,b)$, which gives
$$
z_1=\frac {\beta z_3}b,\ \ 
z_2=\frac {\beta z_4}b,\ \ 
w_1=-\frac {\beta^2z_3}{b^2}\ \ \text{and}\ \ 
w_2=-\frac {\beta^2z_4}{b^2}.
$$
The bracket relations satisfy
	\begin{eqnarray*}
		\lb ZX&=& \beta Y+\frac {\beta z_3}b Z-\frac {\beta^2z_3}{b^2} W,\\
		\lb ZY&=&-\beta X+\frac {\beta z_4}b Z-\frac {\beta^2z_4}{b^2} W,\\
		\lb WX&=& b Y+z_3 Z-\frac {\beta z_3}b W,\\
		\lb WY&=&-b X+z_4 Z-\frac {\beta z_4}b W,\\
		\lb YX&=&\theta_1 Z+\theta_2 W.
	\end{eqnarray*}

\noindent	
$(\mathcal{AK}):$
Since $a=\alpha=0$, this family is {\it almost K\"{a}hler} if and only if $\theta_1=\theta_2=0.$  We get the $4$-dimensional family $\g_{18}^\mathcal{AK}(\beta,b,z_3,z_4)$ with
		\begin{eqnarray*}
			\lb ZX&=& \beta Y+\frac {\beta z_3}b Z-\frac {\beta^2z_3}{b^2} W,\\
			\lb ZY&=&-\beta X+\frac {\beta z_4}b Z-\frac {\beta^2z_4}{b^2} W,\\
			\lb WX&=& b Y+z_3 Z-\frac {\beta z_3}b W,\\
			\lb WY&=&-b X+z_4 Z-\frac {\beta z_4}b W.
		\end{eqnarray*}
		
\noindent
$(\mathcal{I}):$ 
For the family $\g_{18}$ we have 
$$
z_1=\frac {\beta z_3}b, \ \ 
z_2=\frac {\beta z_4}b, \ \ 
w_1=-\frac {\beta^2z_3}{b^2} \ \ \text{and} \ \ 
w_2=-\frac {\beta^2z_4}{b^2}$$ 
and we can easily see that $J$ is {\it integrable} if and only if
$$
(\beta^2-b^2)z_3=2\beta bz_4 \ \ \text{and} \ \ (\beta^2-b^2)z_4=-2\beta bz_3.
$$
We notice that this system has the same form as that of  \eqref{eqg17}. Thus the structure $J$ is integrable if and only if $z_3=z_4=0.$  Here we yield the $4$-dimensional family
$\g_{18}^\mathcal{I}(\beta,b,\theta_1,\theta_2)$, where
		\begin{eqnarray*}
			\lb ZX&=& \beta Y,\\
			\lb ZY&=&-\beta X,\\
			\lb WX&=& b Y,\\
			\lb WY&=&-b X,\\
			\lb YX&=&\theta_1 Z+\theta_2 W.
		\end{eqnarray*}
	
\noindent
$(\mathcal{K}):$
The family is {\it K\"{a}hler} if and only if $\theta_1=\theta_2=0$ and $z_3=z_4=0.$  Hence we obtain the $2$-dimensional family $\g_{18}^\mathcal{K}(\beta,b)$ with
		\begin{eqnarray*}
			\lb ZX&=& \beta Y,\\
			\lb ZY&=&-\beta X,\\
			\lb WX&=& b Y,\\
			\lb WY&=&-b X.
		\end{eqnarray*}
The corresponding simply connected Lie group is clearly a semidirect product $\rn^2\ltimes \rn^2$.
\end{example}


\begin{example}[$\g_{19}(\alpha,\beta,w_1,w_2)$]
\label{exam-F9}
This is a $4$-dimensional family of {\it solvable} Lie algebras. Here $\Lambda=(\alpha,0,\beta,0)$, which implies that
$$z_1=z_2=z_3=z_4=\theta_1=\theta_2=0$$
and the bracket relations are
\begin{eqnarray*}
\lb ZX&=&\alpha X +\beta Y+w_1 W,\\
\lb ZY&=&-\beta X+\alpha Y+w_2 W.
\end{eqnarray*}
	
\noindent
$(\mathcal{AK}):$ 
Here the {\it almost K\"{a}hler} condition is never satisfied since $\alpha\neq 0$ and $\theta_2=0$, hence $\g_{19}^{\mathcal{AK}}=\{0\}$.

\smallskip		
\noindent
$(\mathcal{I}):$ 
Because $z_1=z_2=z_3=z_4=0$, the structure $J$ is {\it integrable} if and only if $w_1=w_2=0.$  We yield the $2$-dimensional family  $\g_{19}^\mathcal{H}(\alpha,\beta)$ given by
\begin{eqnarray*}
\lb ZX&=&\alpha X +\beta Y,\\
\lb ZY&=&-\beta X+\alpha Y.
\end{eqnarray*}
Here the corresponding simply connected Lie group is a semidirect product $\rn^2\ltimes\rn^2$.
	
\smallskip		
\noindent
$(\mathcal{K}):$ 
The {\it K\"{a}hler} condition is never satisfied, so $\g_{19}^{\mathcal{K}}=\{0\}$.
\end{example}


\begin{example}[$\g_{20}(\alpha,a,\beta,w_1,w_2)$]
\label{exam-F10}
Here we have a $5$-dimensional family of {\it solvable} Lie algebras. Now $\Lambda=(\alpha,a,\beta,b)$, which gives
$$
z_1=-\frac {aw_1}\alpha,\ \ 
z_2=-\frac {aw_2}\alpha,\ \ 
z_3=-\frac {a^2w_1}{\alpha^2},\ \ 
z_4=-\frac {a^2w_2}{\alpha^2},$$
$$b=\frac{\beta a}\alpha,\ \ \theta_1=0,\ \ \theta_2=0.$$
The bracket relations fulfill
	\begin{eqnarray*}
		\lb ZX&=&\alpha X+\beta Y-\frac {aw_1}\alpha Z+w_1 W,\\
		\lb ZY&=&-\beta X+\alpha Y-\frac {aw_2}\alpha Z+w_2 W,\\
		\lb WX&=&a X+\frac{\beta a}\alpha Y-\frac {a^2w_1}{\alpha^2} Z+\frac a\alpha w_1 W,\\
		\lb WY&=&-\frac{\beta a}\alpha X+a Y-\frac {a^2w_2}{\alpha^2} Z+\frac a\alpha w_2 W.
	\end{eqnarray*}

\noindent
$(\mathcal{AK}):$ 
Since $\theta_1=\theta_2=0$, $\alpha\neq0$ and $a\neq0$, the {\it almost Kähler} condition is never satisfied, hence $\g_{20}^{\mathcal{AK}}=\{0\}$.

\smallskip
\noindent
$(\mathcal{I}):$ 
Here the coefficients $z_1,\dots,z_4,$ $w_1$ and $w_2$ are the same as for the family $\g_{17}$. Thus the structure $J$ is {\it integrable} if and only if $w_1=w_2=0$. This gives the  $3$-dimensional family $\g_{20}^\mathcal{I}(\alpha,a,\beta)$ with 
		\begin{eqnarray*}
			\lb ZX&=&\alpha X+\beta Y,\\
			\lb ZY&=&-\beta X+\alpha Y,\\
			\lb WX&=&a X+\frac{\beta a}\alpha Y,\\
			\lb WY&=&-\frac{\beta a}\alpha X+a Y.
		\end{eqnarray*}	
Here the corresponding simply connected Lie group is a semidirect product $\rn^2\ltimes\rn^2$.

\smallskip		
\noindent			
$(\mathcal{K}):$ 
The {\it K\"{a}hler} condition is never satisfied, therefore $\g_{20}^{\mathcal{K}}=\{0\}$.
\end{example}


\appendix

\section{The Homogeneous Second Order System}\label{appendix}

In Section \ref{section-the-foliations} we presented the general bracket relations (\ref{system}). To ensure that these actually define a Lie algebra they must satisfy the Jacobi equation.  A standard computation shows that this is equivalent to the following system of 14 second order homogeneous polynomial equations in the 14 variables involved
	\begin{eqnarray*}
		0 &=& \lambda a,\\
		0 &=& \lambda b,\\
		0 &=& -w_2z_3 + w_1z_4 - 2\alpha\theta_1 + rz_1,\\
		0 &=& -2z_4z_1 + 2z_3z_2 - 2a\theta_1 + rz_3,\\
		0 &=& -\lambda z_3 - z_2b + z_4\beta - z_1a + z_3\alpha,\\
		0 &=& -\lambda z_4 - z_2a + z_4\alpha + z_1b - z_3\beta,\\
		0 &=& \lambda\theta_1 - w_1z_4 + w_2z_3 - 2a\theta_2 - rz_1,\\
		0 &=& -\lambda\theta_2 + 2z_1w_2 - 2z_2w_1 - 2\alpha\theta_2 + rw_1,\\
		0 &=& -w_2a - w_1b - z_2\alpha - z_1\beta - \alpha r,\\
		0 &=& -w_2b + w_1a - z_2\beta + z_1\alpha + r\beta,\\
		0 &=& \lambda z_1 - w_2b - z_2\beta - w_1a - z_1\alpha,\\
		0 &=& z_2a + z_1b - z_4\alpha - z_3\beta - ar,\\
		0 &=& z_2b - z_1a - z_4\beta + z_3\alpha + rb,\\
		0 &=& \lambda z_2 - w_2a - z_2\alpha + w_1b + z_1\beta.
	\end{eqnarray*}


\end{document}